\documentclass{article}

\usepackage{amsmath, mathrsfs, amssymb, stmaryrd, cancel, hyperref, relsize,tikz,amsthm}
\usepackage{graphicx}
\usepackage{xfrac}
\hypersetup{pdfstartview={XYZ null null 1.25}}
\usepackage[all]{xy}
\usepackage[normalem]{ulem}

\theoremstyle{plain}
\newtheorem{theorem}{Theorem}[section]{\bfseries}{\itshape}
\newtheorem{proposition}[theorem]{Proposition}{\bfseries}{\itshape}
\newtheorem{definition}[theorem]{Definition}{\bfseries}{\upshape}
\newtheorem{lemma}[theorem]{Lemma}{\bfseries}{\upshape}
\newtheorem{example}[theorem]{Example}{\bfseries}{\upshape}
\newtheorem{corollary}[theorem]{Corollary}{\bfseries}{\upshape}
\newtheorem{remark}[theorem]{Remark}{\bfseries}{\upshape}

\newcommand{\bw}{\bigwedge}
\newcommand{\bv}{\bigvee}

\newcommand{\gA}{\mathfrak{A}}
\newcommand{\sL}{\mathscr{L}}
\newcommand{\sC}{\mathscr{C}}
\newcommand{\pu}{\mathit{p}^\uparrow}

\newcommand{\abf}{(\alpha,\beta)\text{-filter}}
\newcommand{\abr}{(\alpha,\beta)\text{-representable}}
\newcommand{\abe}{(\alpha,\beta)\text{-embedding}}
\newcommand{\cF}{\mathcal{F}}

\newcommand{\bbA}{\mathbb{A}}
\newcommand{\bbS}{\mathbb{S}}

\newcommand{\fin}{\mathbf{small}}
\newcommand{\esS}{\mathit{s}^{\bbS}}
\newcommand{\etS}{\mathit{t}^{\bbS}}
\newcommand{\euS}{\mathit{u}^{\bbS}}
\newcommand{\eaA}{\mathit{a}^{\bbA}}
\newcommand{\ebA}{\mathit{b}^{\bbA}}
\newcommand{\ecA}{\mathit{c}^{\bbA}}
\newcommand{\edA}{\mathit{d}^{\bbA}}
\newcommand{\lb}{\text{\textbf{glb}}}
\newcommand{\ub}{\text{\textbf{lub}}}

\newcommand{\boU}{\text{\textbf{U}}}
\newcommand{\boA}{\text{\textbf{A}}}

\newcommand{\bfC}{\text{\textbf{C}}}

\title{Representable posets}
\author{Rob Egrot}
\date{}

\begin{document}
\maketitle

\begin{abstract}
A poset is representable if it can be embedded in a field of sets in such a way that existing finite meets and joins become intersections and unions respectively (we say finite meets and joins are preserved). More generally, for cardinals $\alpha$ and $\beta$ a poset is said to be $(\alpha,\beta)$-representable if an embedding into a field of sets exists that preserves meets of sets smaller than $\alpha$ and joins of sets smaller than $\beta$. We show using an ultraproduct/ultraroot argument that when $2\leq\alpha,\beta\leq \omega$ the class of $(\alpha,\beta)$-representable posets is elementary, but does not have a finite axiomatization in the case where either $\alpha$ or $\beta=\omega$. We also show that the classes of posets with representations preserving either countable or \emph{all} meets and joins are pseudoelementary. 
\let\thefootnote\relax\footnotetext{Author: Rob Egrot, Faculty of ICT, Mahidol University,
999 Phuttamonthon 4 Road,
Salaya,
Nakhon Pathom 73170,
Thailand,
email: robert.egr@mahidol.ac.th
}
\end{abstract}

\begin{section}{Introduction}
Every poset $P$ can be thought of as a set of sets ordered by inclusion by considering the embedding $P\to \wp(P)$ defined by $p\mapsto \{q\in P:q\leq p\}$. Indeed, this representation has the advantage of preserving all existing meets, finite and infinite. Similarly the representation defined by $p\mapsto \{q\in P:q\not\geq p\}$ preserves existing joins \cite{Abi68}. It follows from this that semilattices can be represented in such a way that their binary operation is modelled by union or intersection appropriately, though in general we cannot construct representations where both existing joins and meets are interpreted as unions and intersections
respectively. 

Distributivity, or the lack of it, is the issue here, as might be expected given that fields of sets are distributive. In the case of Boolean algebras, which are necessarily distributive, representability follows as an easy corollary of Stone's theorem \cite{Stone37}. For lattices distributivity alone is a necessary and sufficient condition for representability.

\begin{definition}[Lattice representation]
Let $L$ be a bounded lattice. A \emph{representation} of $L$ is a bounded lattice embedding $h\colon L\to\wp(X)$ for some set $X$, where $\wp(X)$ is considered as a field of sets, under the operations of set union and intersection. When such a representation exists we say that $L$ is \emph{representable}.
\end{definition}

\begin{theorem}[Birkhoff]\label{T;PrimeIdeal}
Let $L$ be a distributive lattice, let $F$ be a filter of $L$, and let $I$ be an ideal of $L$ with $F\cap I=\emptyset$. Then there is a prime filter $F'\subset L$ with $F\subseteq F'$ and $I\cap F'=\emptyset$.
\end{theorem}

As a corollary to this we have the following result.

\begin{theorem}\label{T;Lrep}
A bounded lattice is representable if and only if it is distributive.
\end{theorem}
\begin{proof}
If $L$ is a distributive lattice and $K$ is the set of prime filters of $L$ it's not difficult to show that the map $h\colon P \to \wp(K)$ defined by $h(a)=\{f\in K:a\in f\}$ is a representation using theorem \ref{T;PrimeIdeal}. The converse is trivial. 
\end{proof}

In the meet-semilattice case we have representability for existing finite joins if and only if an infinite family of first order axioms demanding that whenever $a\wedge(b_1\vee...\vee b_n)$ is defined, then $(a\wedge b_1)\vee...\vee (a\wedge b_n)$ is also defined and the two are equal is satisfied \cite{Bal69,Sch72}. In the lattice case of course only a single distributivity equation is both necessary and sufficient. It is possible that in the meet-semilattice case meets could distribute over some finite cardinalities of existing joins but not others, which would necessitate a family of axioms to axiomatize representability. Schein \cite{Sch72} asserts that this is the case, and indeed that an infinite family is required, but there is some reason to doubt that this has been established as a fact\footnote{Post-publication note: Actually this was conclusively resolved by \cite{Kea97}, where it is shown that the class of representable semilattices is not finitely axiomatizable}. We expand on this in the discussion following corollary \ref{C;compRep}. 

It is tempting to try to apply the meet-semilattice axiom schema directly to posets. There is a problem however, as the classical Zorn's lemma based argument for semilattices and lattices does not work in the poset case. A step in this argument assumes the existence of finite meets, which is invalid in the more general setting of posets. This is a symptom of a deeper problem, as a simple generalization of the axiom schema for semilattices (such as appears as the condition LMD in \cite{CheKem92}, see definition \ref{D;LMD} below) is not expressive enough for the poset situation (see example \ref{E;express}, where it is shown that LMD is not a necessary condition for a poset to have a representation). Note that the question of whether LMD is a \emph{sufficient} condition for representability with respect to all meets and binary joins is raised as open in \cite{CheKem92}, and appears not to have been resolved.

\begin{definition}[LMD]\label{D;LMD}
A poset $P$ is $LMD$ if and only if, for all $x,y,z\in P$, whenever $x\wedge (y\vee z)$ is defined then $(x\wedge y)\vee (x\wedge z)$ is also defined and they are equal.
\end{definition}

\begin{example}[The condition $LMD$ is not necessary for representability]\label{E;express} Let $P$ be the poset 
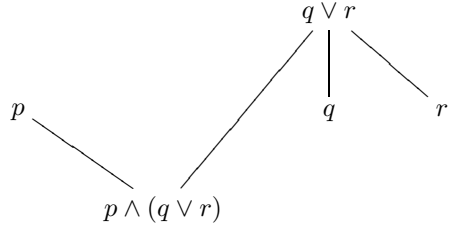
\[\xymatrix{& & q\vee r\ar@{-}[dr]\ar@{-}[d]\ar@{-}[ddl] \\
p\ar@{-}[rd] & & q & r \\
& p\wedge(q\vee r)
}\] 
Then $p\wedge q$, $p\wedge r$, and $(p\wedge q)\vee (p\wedge r)$ do not exist in $P$ but $P$ is representable. For example using the field of sets generated by $\{a,c,d\},\{b,d\},\{b,c\}$, where $p\mapsto \{a,c,d\}$, $q\mapsto \{b,d\}$ and $r\mapsto \{b,c\}$.
\end{example}

Suppose we were to weaken $LMD$ so that equality is only necessary in the case where both sides are defined (see definition \ref{D;D2} below)? It turns out that an axiom schema along these lines is not sufficient to ensure representability. We will demonstrate this in example \ref{E;DnotRep} later. 
\begin{definition}[$\bar{D}$2]\label{D;D2}
We say a poset $P$ satisfies condition $\bar{D}2$ if and only if for all $x,y,z\in P$, whenever $x\wedge(y\vee z)$ and $(x\wedge y)\vee (x\wedge z)$ are defined then they are equal.
\end{definition}

Despite the failure of the most natural candidates, it is nevertheless the case that the class of representable posets is elementary, though an infinite number of axioms are needed. In section \ref{S;reps} we set down a general framework for representations of posets and prove a characterization theorem (theorem \ref{T;rep}). In section \ref{S;ultras} we provide some well known definitions and results from model theory which we shall use in later sections. In section \ref{S;elem} we use this characterization in a closure under ultraproducts/ultraroots argument to demonstrate that the class of representable posets is elementary. We also show that an infinite number of axioms is necessary. Finally in section \ref{S;comp} we show that the notion of complete representability introduced in section \ref{S;reps} is pseudoelementary, and describe some situations where it is known not be elementary.

\end{section}

\begin{section}{Representations for posets}\label{S;reps}
In the following definitions we assume that $\alpha$ and $\beta$ are cardinals with $2<\alpha,\beta$.

We observe the following notational convention. For any subset $S$ of a poset $P$ we write $S^\uparrow$ for $\{p\in P: \exists s\in S,\; s\leq p\}$. For $p\in P$ we write $\pu$ as shorthand for $\{p\}^\uparrow$ when the meaning is not ambiguous.

\begin{definition}[$(\alpha,\beta)$-morphism]\label{D;mapdefs}
Given posets $P_1$ and $P_2$ we say a map $f\colon P_1\to P_2$ is an $(\alpha,\beta)$-morphism if $f(\bw_I p_i)=\bw_I f(p_i)$ whenever $|I|<\alpha$ and $\bw_I p_i$ is defined, and $f(\bv_I p_i)=\bv_I f(p_i)$ whenever $|I|<\beta$ and $\bv_I p_i$ is defined. If $f$ is also an order embedding we say it is an $(\alpha,\beta)$-embedding (note that $f$ will always be order preserving).

If for some $\alpha$ a morphism $f$ is an $(\alpha,\beta)$-morphism for all $\beta$ we say it is an $(\alpha,C)$-morphism, and similarly if for some $\beta$ it is an $(\alpha,\beta)$-morphism for all $\alpha$ we say it is a $(C,\beta)$-morphism. If $f$ is an $(\alpha,\beta)$-morphism for all $\alpha,\beta$ then we say it is a $C$-morphism, or a complete morphism. We make the appropriate definitions for complete embeddings etc. When $\alpha=\beta$ we just say e.g. $\alpha$-morphism.
\end{definition}

\begin{definition}[$(\alpha,\beta)$-representation]\label{D;rep}
An $(\alpha,\beta)$-representation of a poset $P$ is an $(\alpha,\beta)$-embedding $h\colon P \to \wp(X)$ for some set $X$ where $\wp(X)$ is considered as a field of sets. When $P$ has a top and/or bottom, we demand that $h$ maps them to $X$ and/or $\emptyset$ respectively. When $P$ has an $(\omega,\omega)$-representation we omit the modifier and simply say it has a \emph{representation}, or that it is \emph{representable}. Similar definitions are made for complete representations etc. using definition \ref{D;mapdefs} above. When $\alpha=\beta$ we just say e.g. $\alpha$-representation.
\end{definition}

The concept of representability is equivalent to a separation property. We make this precise in theorem \ref{T;rep}, though first we require some preliminary definitions.

\begin{definition}[$\abf$]\label{D;wfilter}
$S\subseteq P$ is an $\abf$ of $P$ if it is closed upwards and for all $\emptyset\subset X\subseteq S$ with $|X|<\alpha$ we have $\bw X\in S$ whenever $\bw X$ is defined, and whenever $\emptyset\subset Y\subseteq P$ and $|Y|<\beta$ with $\bv Y$ defined in $P$ we have $\bv Y\in S\implies y\in S$ for some $y\in Y$. $(C,C)$-filters etc. are defined similarly. When $\alpha=\beta$ we just say e.g. $\alpha$-filter.
\end{definition}  

We define \emph{$(\alpha,\beta)$-ideals} dually as being the sets that are $\abf$s in the order dual $P^\delta$. Note that in a lattice the $\abf$s and $(\alpha,\beta)$-ideals are precisely the prime filters and prime ideals (for all $2<\alpha,\beta\leq \omega$).

\begin{proposition}
If $F$ is an $\abf$ of $P$ then $F'=P\setminus F$ is a $(\beta,\alpha)$-ideal of $P$.
\end{proposition}
\begin{proof}
We must show that $F'$ is a $(\beta,\alpha)$-filter of $P^\delta$. Clearly $F'$ is up-closed in $P^\delta$ so let $X\subseteq F'$ with $|X|<\beta$ and suppose $\bw X$ exists in $P^\delta$. Then if $\bw X\notin F'$ (in $P^\delta$) then $\bv X\in F$ (in $P$), and so we must have $x\in X\cap F$, which would be a contradiction. Similarly, suppose $Y\subset P^\delta$ with $|Y|<\alpha$ and that $\bv Y$ exists in $P^\delta$. Then if for all $y\in Y$ we have $y\notin F'$ then we must have $Y\subseteq F$. But then $\bw Y\in F$ (in $P$), and so $\bv Y \notin F'$ (in $P^\delta$). So $F'$ is a $(\beta,\alpha)$-filter of $P^\delta$ as required. 
\end{proof}

\begin{corollary}\label{C;FI}
$F$ is an $\abf$ of $P$ $\iff F'$ is a $(\beta,\alpha)$-ideal of $P$
\end{corollary}
\begin{proof}
The forward direction is the preceding proposition, and the backward direction follows from the dual statement.
\end{proof}

\begin{definition}[Separating]
$S\subseteq \wp(P)$ is separating over $P$ if whenever $p\not\leq q\in P$ there is $X\in S$ with $p\in X$ and $q\notin X$. We say $S$ is \emph{dually-separating} over $P$ if whenever $p\not\leq q$ there is $X\in S$ with $q\in X$ and $p\notin X$.
\end{definition}

\begin{theorem}\label{T;rep}
For a poset $P$, and for all cardinals $2<\alpha,\beta$ the following are equivalent:
\begin{enumerate}
\item The $\abf$s of $P$ are separating over $P$,
\item $P$ is $(\alpha,\beta)$-representable,
\end{enumerate}
\end{theorem}
\begin{proof}
\mbox{}
\begin{enumerate}
\item[]$1.\Rightarrow 2.$ Let $\cF$ be the set of $\abf$s of $P$ and define $h$ by $h(p)=\{\gamma\in \cF:p\in \gamma\}$. If $\cF$ is separating over $P$ then $h$ will certainly be an order embedding. Moreover, if $X\subseteq P$ with $|X|<\alpha$ and $\bw X$ defined in $P$ then clearly $h(\bw X)\subseteq \bigcap h[X]$ and if there is an $\abf$ $\gamma$ with $x\in \gamma$ for all $x\in X$ then $\bw X \in \gamma$ too, so $h(\bw X)=\bigcap h[X]$. Similarly, if $Y\subseteq P$ with $|Y|<\beta$ and $\bv Y$ defined in $P$, then clearly $h(\bv Y)\supseteq \bigcup h[Y]$, and if there is an $\abf$ $\gamma$ with $\bv Y\in \gamma$ then there must be $y\in Y$ with $y\in\gamma$ so $h(\bv Y)=\bigcup h[Y]$. 
\item[]$2.\Rightarrow 1.$ Let $h\colon P\to \wp(X)$ be an $(\alpha,\beta)-$representation and let $a \not\leq b\in P$. Then there is $x\in X$ with $x\in h(a)$ and $x\notin h(b)$. Consider $h^{-1}[x]=\{p\in P: x\in h(p)\}$. Then clearly $a\in h^{-1}[x]$ and $b\notin h^{-1}[x]$. We will show that $h^{-1}[x]$ is an $\abf$. If $p\in h^{-1}[x]$ and $q\geq p$ then $q\in h^{-1}[x]$ as $h(p)\subseteq h(q)$ so $h^{-1}[x]$ is up-closed. Let $S\subseteq h^{-1}[x]$ with $|S|<\alpha$ and suppose $\bw S$ is defined in $P$. By definition we have $x\in h(s)$ for all $s\in S$, so $x\in \bigcap h[S]=h(\bw S)$, and thus $\bw S \in h^{-1}[x]$. Similarly, if $S\subseteq P$ with $|S|<\beta$ and $\bv S$ defined and in $h^{-1}[x]$ then by definition $x\in h(\bv S)=\bigcup h[S]$ so $x\in h(s)$ for some $s\in S$ and thus $s\in h^{-1}[x]$ for some $s\in S$. So $h^{-1}[x]$ is an $\abf$ and we are done.
\end{enumerate}
\end{proof}
There is of course a dual version of this theorem with ideals in place of filters. There are also versions of this result for $(C,\beta)$, $(\alpha, C)$ and $C$-filters obtainable by simply substituting $C$ consistently. The proofs are essentially the same.

\begin{corollary}
When $\alpha,\beta\leq\omega$ a poset $P$ is $\abr$ if and only if there is a distributive lattice $L$ and an $\abe$ $e\colon P \to L$.
\end{corollary}
\begin{proof}
If $P$ is representable then the representation provides a suitable distributive lattice. Conversely, if $L$ is a distributive lattice and $e\colon P\to L$ is an $\abe$ then, by theorem \ref{T;PrimeIdeal}, whenever $p\not\leq q$ in $P$ there is a prime filter of $L$ containing $e(p)$ but not $e(q)$. The preimage of this filter under $e$ is an $\abf$ so the $\abf$s of $P$ are separating over $P$ and thus $P$ is representable. 
\end{proof}

\begin{corollary}\label{C;compRep}
For a poset $P$ the following are equivalent:
\begin{enumerate}
\item $P$ is completely representable.
\item There is a complete lattice $L$ with the following properties:
\begin{enumerate}
\item the completely join-irreducibles of $L$ are join-dense in $L$, 
\item $L$ satisfies the distributivity property that $p\wedge \bv X=\bv_X\{p\wedge x\}$ for all $\{p\}\cup X\subseteq L$,
\item there is a complete embedding $e\colon P\to L$.
\end{enumerate}
\item There is a complete lattice $L$ with the following properties:
\begin{enumerate}
\item the completely meet-irreducibles of $L$ are meet-dense in $L$, 
\item $L$ satisfies the distributivity property that $p\vee \bw X=\bw_X\{p\vee x\}$ for all $\{p\}\cup X\subseteq L$,
\item there is a complete embedding $e\colon P\to L$.
\end{enumerate}
\end{enumerate}
\end{corollary}
\begin{proof}
We prove $1.\iff 2.$ and the rest follows from order duality. The forward direction is trivial. For the converse note that by the distributivity property of $L$ if $p\in L$ is completely join-irreducible then it also has the property that whenever $p\leq \bv X$ there must be $x\in X$ with $p\leq x$ (this property is usually known as complete join-primality). Thus the completely join-irreducibles of $L$ generate a separating set of $C$-filters in $L$, and the preimages of these under $e$ is a separating set of $C$-filters of $P$.  
\end{proof}

\begin{remark}
In the preceding result note that 2(a) and 2(b) together imply not only that $L$ is an algebraic frame, but also that it is a co-algebraic co-frame. In fact, it follows from theorem \ref{T;rep} that $L$ is completely representable, and thus that it is completely distributive. Moreover, since $L$ is complete, complete representability also implies that the completely meet-prime elements are meet-dense in $L$. So any complete lattice satisfying 2(a) and 2(b) will not only satisfy 3(a) and 3(b) (and vice versa by duality), but also the superficially stronger condition of being completely distributive, join-generated by its completely join-primes, and meet-generated by its completely meet-primes.
\end{remark}

If $\alpha\leq \alpha'$ and $\beta\leq \beta'$ then the fact that $P$ is $\abr$ whenever it is $(\alpha',\beta')$-representable comes straight from the definition. The converse however is not true in general, as we see in examples \ref{E;not4} and \ref{E;notN} below. Note that \cite{Sch72} claims implicitly that meet-semilattices can be constructed where the partial join operation is $m$-representable but not $n$-representable (with $m<n\leq\omega$), but does not provide one. Many years later the existence of such semilattices was raised as an open question at the end of \cite{PawTha80}. It is shown in \cite{SCLS85} that any such semilattice example cannot be finite. An earlier paper \cite{HooShu84} had claimed that no such semilattice could exist, and thus that for semilattices $n$-representability implies $\omega$-representability for all $n<\omega$. However, the MathSciNet review of this article notes that the paper makes essential use of the original, slightly erroneous form of \cite[theorem 2.2]{Bal69}, and hence the result is not considered to be correct. A corrected version of \cite[theorem 2.2]{Bal69} can be found as \cite[theorem 2.2]{HooShu82}. 

\begin{example}[A poset that is $3$-representable but not $4$-representable]\label{E;not4}
Consider the following poset $P$:
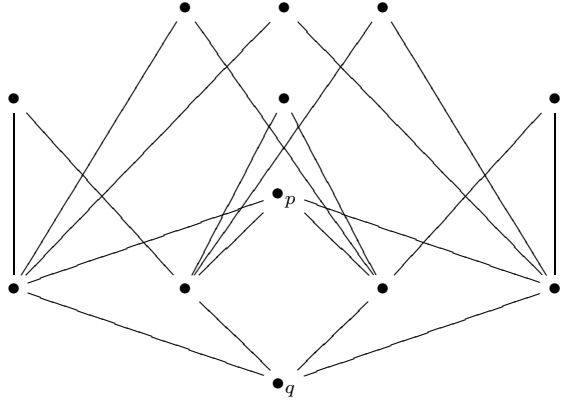
\[\xymatrix{
& & \bullet\ar@{-}[dddll]\ar@{-}[dddrr]  &  \bullet\ar@{-}[dddlll]\ar@{-}[dddrrr]  & \bullet\ar@{-}[dddll]\ar@{-}[dddrr]  \\
\bullet & & & \bullet & & & \bullet\\
&&& \bullet_p\ar@{-}[dlll]\ar@{-}[dl]\ar@{-}[dr]\ar@{-}[drrr]\\
\bullet\ar@{-}[uu]& & \ar@{-}[uull]\bullet\ar@{-}[uur] & & \ar@{-}[uul]\bullet\ar@{-}[uurr] & & \ar@{-}[uu]\bullet \\
& & & \bullet_q\ar@{-}[ulll]\ar@{-}[ul]\ar@{-}[ur]\ar@{-}[urrr] 
}\] 
Then no non-trivial binary joins are defined in $P$, so we can find a $3$-representation for $P$ by taking the principal up-sets and appealing to theorem \ref{T;rep}. However, there is no $4$-filter containing $p$ but not $q$, so by the same theorem $P$ has no $4$-representation.
\end{example} 

Following the suggestion of the referee we generalize example \ref{E;not4} to find posets that are $(n-1)$-representable but not $n$-representable for arbitrary $n<\omega$.

\begin{example}[A poset that is $(n-1)$-representable but not $n$-representable]\label{E;notN}
Let $X=\{x_1,...,x_n\}$ be a set with $|X|=n$. For every $s\subset X$ with $|s|= n-2$ define a new element $y_s$. Let $Y$ be the set of all the $y_s$ elements. Let $p$ and $q$ not occur in either $X$ or $Y$. We define $P_n$ by extending the trivial order on $X\cup Y\cup\{p,q\}$ as follows:
\begin{enumerate}
\item $q < x$ for all $x\in X$
\item $p > x$ for all $x\in X$
\item $y_s > x \iff x\in s$
\end{enumerate}
and closing for transitivity. Then $P_n$ is $m$-representable for all $m<n$ because there are no non-trivial joins of cardinality less than or equal to $m-1$ defined in $P$. To see this let $t\subset X$ with $2\leq|t|<m$ and define $S_t=\{s\subset X: t\subseteq s$ and $|s|=n-2\}$. Then the upper bounds of $t$ are $\{y_s:s\in S_t\}\cup\{p\}$, and this set is an antichain. So we can once again take the set of principal up-sets and appeal to theorem \ref{T;rep}. However, $p$ is the only upper bound for $X\setminus\{x_1\}$, so any $n$-filter containing $p$ must also contain $x$ for some $x\in X\setminus\{x_1\}$. Similarly any such $n$-filter must also contain $x'$ for some $x'\in X\setminus\{x\}$, and the meet of $x$ and $x'$ is $q$. So there can be no $n$-filter containing $p$ but not $q$ and thus $P_n$ cannot be $n$-representable.  
\end{example}

Note that the posets in example \ref{E;notN} are actually $(\alpha,(n-1))$-representable but not $(\alpha,n)$-representable for all $\alpha>2$.

The following example demonstrates that a weaker axiom schema along the lines of $\bar{D}2$ from definition \ref{D;D2} is not sufficient for our needs.
\begin{example}[A poset satisfying $\bar{D}2$ that fails to be $3$-representable]\label{E;DnotRep}
Let $P$ be the poset in the diagram below. Then there is no $3$-filter containing $p$ but not $q$.
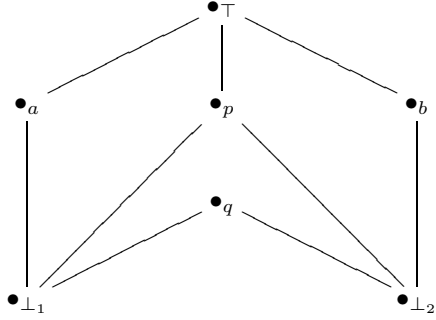
\[\xymatrix{ & & \bullet_\top\ar@{-}[dll]\ar@{-}[d]\ar@{-}[drr] \\
\bullet_a\ar@{-}[dd] & & \ar@{-}[ddll]\bullet_p\ar@{-}[ddrr] & & \bullet_b\ar@{-}[dd] \\
& &  \bullet_q\ar@{-}[dll]\ar@{-}[drr]  \\
\bullet_{\bot_1} & & & & \bullet_{\bot_2}
}\]

We must show that $P$ satisfies $\bar{D}2$. We shall show that for all possible choices of $x,y,z\in P$, either one or both of $x\wedge(y\vee z)$ and $(x\wedge y)\vee (x\wedge z)$ is undefined or they are equal. To this end we will use the fact (easily checkable) that in any poset if $x,y,z$ satisfy any of the following conditions then $\bar{D}2$ holds for that particular choice of $x,y,z$:
\begin{enumerate}
\item $x\leq y$ or $x\leq z$ \label{C1}
\item $y\leq x$ and $z\leq x$ \label{C2}
\item $y\leq z$ or $z\leq y$ \label{C3}
\end{enumerate}
Now we show that there is no element of $P$ that makes a suitable choice for $x$ in a counterexample to $\bar{D}2$:
\begin{enumerate}
\item[]$x\neq\top$. This follows directly from \ref{C2}.
\item[]$x\neq \bot_1$ (and $x\neq \bot_2$ by symmetry). If $x=\bot_1$ then in order for $x\wedge y$ to be defined we must have $x\leq y$ and so we apply \ref{C1}.
\item[]$x\neq q$. If $x=q$ then in order for $x\wedge y$ to be (non-trivially) defined we must have one of $y=a,b,\bot_1$, or $\bot_2$. Similar considerations apply to $z$ and so by \ref{C3} we must have (wlog) $y=a$ or $\bot_1$, and $z=b$ or $\bot_2$. But then $x\wedge y = \bot_1$ and $x\wedge z = \bot_2$ and so $(x\wedge y)\vee (x\wedge z)$ is undefined.
\item[]$x\neq a$ (and $x\neq b$ by symmetry). If $x= a$ then in order for $x\wedge y$ to be (non-trivially) defined we must have one of $y=\bot_1,q,p$, or $\top$. We can rule out $y=\top$ with \ref{C1}, and if $y=\bot_1$ then if $x\wedge z$ is defined we will have $y\leq z$ and we would apply \ref{C3}. So we must have either $y=p$ or $q$. Similar considerations apply to $z$ and so wlog suppose $y=p$ and $z=q$. Then $y\vee z$ is not defined and we are done.
\item[]$x\neq p$. If $x=p$ then in order for $x\wedge y$ to be (non-trivially) defined we must have one of $y=\bot_1,\bot_2,a,b$, or $\top$. By \ref{C1} we can exclude $y=\top$ and so applying \ref{C3} suppose wlog that $y=a$ or $\bot_1$, and $z=b$ or $\bot_2$. But then as in the case where $x=q$ we conclude that $(x\wedge y)\vee (x\wedge z)$ is not defined. 
\end{enumerate}
\end{example}

To close this section recall theorem \ref{T;PrimeIdeal}, and the similar result for semilattices \cite[theorem 2.2]{Bal69},\cite[theorem 2.2]{HooShu82}. An obvious question is whether analogous results hold for representable posets. For example, if $P$ is $3$-representable, $p\in P$, and $\gamma$ is an up-closed subset of $P$ closed under existing binary meets with $p\notin \gamma$, can $\gamma$ always be extended to a $3$-filter of $P$ that excludes $p$? The answer turns out to be no, as we see in example \ref{E;primeIdeal} below. 

\begin{example}\label{E;primeIdeal}
Consider the following system of sets $P$:\\
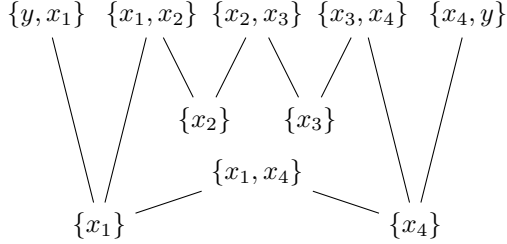
\begin{tikzpicture}[scale=.7]
\node (q1) at (0,-3) {$\{x_1,x_4\}$};
\node (x1) at (-4,0) {$\{y,x_1\}$};
\node (x2) at (-2,0) {$\{x_1,x_2\}$};
\node (x3) at (0,0) {$\{x_2,x_3\}$};
\node (x4) at (2,0) {$\{x_3,x_4\}$};
\node (x5) at (4,0) {$\{x_4,y\}$};
\node (y1) at (-3,-4) {$\{x_1\}$};
\node (y2) at (-1,-2) {$\{x_2\}$};
\node (y3) at (1,-2) {$\{x_3\}$};
\node (y4) at (3,-4) {$\{x_4\}$};

\draw (x1)--(y1)--(x2);
\draw (x2)--(y2)--(x3);
\draw (x3)--(y3)--(x4);
\draw (x4)--(y4)--(x5);
\draw (y1)--(q1)--(y4);

\end{tikzpicture}

\noindent
Then $P$ is (completely) represented by itself. Moreover, \begin{equation*}\gamma=\{\{y,x_1\},\{x_2,x_3\},\{x_4,y\}\}\end{equation*} is an up-closed set that is closed under existing binary meets and does not contain $\{x_1,x_4\}$. However, every $3$-filter of $P$ that contains $\gamma$ must contain $\{x_1,x_4\}$.
\end{example}
\end{section}

\begin{section}{Ultraproducts and ultraroots}\label{S;ultras}
This section serves to introduce the aspects of model theory on which we base the results of the following sections.

\begin{definition}[Ultraproduct]\label{D;ultraprod}
Given a language $\sL$, an indexing set $I\neq\emptyset$, a set of $\sL$-structures $\{\gA_i:i\in I\}$, and an ultrafilter $U\subset \wp(I)$, we define an equivalence relation $\sim$ on $\prod_I \gA_i$ by $x\sim y\iff \{i\in I:x(i)=y(i)\}\in U$. We define the ultraproduct of $\{\gA_i:i\in I\}$ over $U$ to be the set of $\sim$ equivalence classes of $\prod_I \gA_i$ and we denote it with $\prod_U\gA_i$. We can use $\prod_U\gA_i$ as an $\sL$-structure by making interpretations as follows:
\begin{itemize}
\item $\prod_U\gA_i\models R([x_1],...,[x_n])\iff\{i\in I:\gA_i\models R_i(x_1(i),...,x_n(i))\}\in U$ for all $n$-ary relation symbols $R$ of $\sL$, where $R_i$ is the interpretation of that relation in $\gA_i$.
\item $f([x_1],...,[x_n])=[(f_i(x_1(i),...,x_n(i)))_I]$ for all $n$-ary function symbols $f$ of $\sL$, where $f_i$ is the interpretation of that function in $\gA_i$ and $(f_i(x_1(i),...,x_n(i)))_I$ is the element of $\prod_I \gA_i$ obtained by applying $f_i$ to $(x_1(i),...,x_n(i))$ for each $i\in I$.
\item $c=[\bar{c}]$, where $\bar{c}(i)$ is the interpretation of $c$ in $\gA_i$ for all $i\in I$ for all constant symbols $c$ of $\sL$.
\end{itemize}
It's easy to check that this interpretation defines an $\sL$-structure (see e.g. \cite[Proposition 4.1.7]{Chakei90}). When $\gA_i=\gA_j=\gA$ for all $i,j\in I$ we say $\prod_U\gA_i$ is the \emph{ultrapower} of $\gA$ over $U$, and we write $\prod_U\gA$. In this case we say $\gA$ is an \emph{ultraroot} of $\prod_U\gA$.
\end{definition}

\begin{theorem}[\L o\'s]\label{T;Los}
Let $\sL$ be a language, let $\phi$ be an $\sL$-sentence, let $I\neq\emptyset$ be an indexing set, let $\{\gA_i:i\in I\}$ be a set of $\sL$-structures, and let $U$ be an ultrafilter of $\wp(I)$. Then $\prod_U\gA_i\models \phi\iff \{i\in I:\gA_i\models \phi\}\in U$.
\end{theorem}
Theorem \ref{T;Los} is sometimes referred as the fundamental theorem of ultraproducts (see e.g. \cite[Theorem 4.1.9]{Chakei90} for a proof). It allows axiomatization results to be proved with largely algebraic methods. In particular we shall use the following theorem based on the deep results of Shelah \cite{She71}.

\begin{theorem}\emph{(\cite[Theorem 3.32]{HirHod02})\textbf{.}}\label{T;elemcond}\label{T;elemUltra}
Given a language $\sL$ and a class $\sC$ of $\sL$-structures, $\sC$ is elementary if and only if $\sC$ is closed under taking isomorphic copies, ultraproducts, and ultraroots.
\end{theorem}

\end{section}

\begin{section}{Elementary classes of representable posets}\label{S;elem}
In this section $2<\alpha,\beta\leq \omega$. The results do not necessarily hold for larger values.
\begin{lemma}\label{L;filters}
Let $I$ be an indexing set, let $P_i$ be a poset for each $i\in I$, let $U$ be an ultrafilter of $\wp(I)$, and let $u\in U$. Then if $\gamma_i$ is an $\abf$ of $P_i$ for each $i\in u$ then $\Gamma=\{[x]\in \prod_U P_i: \{i:x(i)\in\gamma_i\}\in U\}$ is an $\abf$ of $\prod_U P_i$. 
\end{lemma}
\begin{proof}
Let $[x]\in \Gamma$, and let $[y]\geq [x]$. Then let $u_1=\{i:x(i)\in\gamma_i\}\in U$, and $u_2=\{i:y(i)\geq x(i)\}\in U$, so $\{i:y(i)\in \gamma_i\}\supseteq u_1\cap u_2\in U$ and thus $[y]\in \Gamma$. So $\Gamma$ is up-closed.

Now let $[x],[y]\in \Gamma$. A similar argument to the above proves that $[x]\wedge[y]\in\Gamma$. This argument generalizes to arbitrary finite meets.

Finally let $[x]\vee[y]\in \Gamma$, and let $v=\{i:x(i)\vee y(i)\in \gamma_i\}\in U$. Let $u_1=\{i:x(i)\in \gamma_i\}$, and $u_2=\{i:y(i)\in \gamma_i\}$. Then $u_1\cup u_2=v$ and so either $u_1\in U$ or $u_2\in U$. Thus either $[x]\in \Gamma$ or $[y]\in \Gamma$. This argument generalizes to arbitrary finite joins, and so we conclude that $\Gamma$ is an $\abf$.
\end{proof}

\begin{proposition}
The class of $(\alpha,\beta)$-representable posets is closed under ultraproducts.
\end{proposition}
\begin{proof}
Let $I$ be an indexing set, let $P_i$ be an $(\alpha,\beta)$-representable poset for each $i\in I$, let $[a],[b]\in \prod_U P_i$ and let $[a]\not\leq[b]$. Let $u=\{i:a(i)\not\leq b(i)\}\in U$. Since $P_i$ is $(\alpha,\beta)$-representable for each $i\in u$ there is an $\abf$ $\gamma_i$ with $a(i)\in \gamma_i$ and $b(i)\notin \gamma_i$. By lemma \ref{L;filters} $\Gamma=\{[x]\in \prod_U P_i: \{i:x(i)\in\gamma_i\}\in U\}$ is an $\abf$ of $\prod_U P_i$, and clearly $[a]\in \Gamma$. Since $\{i:b(i)\not\in\gamma_i\}= u\in U$ we also have $[y]\not\in\Gamma$ so we are done.
\end{proof}

\begin{lemma}\label{L;filterFromPower}
Let $P$ be a poset, let $I$ be an indexing set, let $U$ be an ultrafilter of $\wp(I)$, and let $\Gamma$ be an $\abf$ of $\prod_U P$. Given $p\in P$ define $\bar{p}\in \prod_I P$ by $\bar{p}(i)=p$ for all $i\in I$. Then $\gamma=\{p\in P:[\bar{p}]\in \Gamma\}$ is an $\abf$ of $P$.
\end{lemma}
\begin{proof}
Let $p\in P$. Then $p\in \gamma\iff [\bar{p}]\in \Gamma$. Suppose $p\in\gamma$ and let $q\geq p$. Then clearly $[\bar{q}]\geq[\bar{p}]$ and so $[\bar{q}]\in\Gamma$. This means $q\in \gamma$ by definition and so $\gamma$ is up-closed. 

Now suppose $p,q\in \gamma$ and that $p\wedge q$ exists in $P$. Then $[\bar{p}],[\bar{q}]\in \Gamma$, and so $[\bar{p\wedge q}]\in \Gamma$ which means $p\wedge q \in \gamma$. This argument generalizes to arbitrary finite meets, and a similar argument applies to joins and so $\gamma$ is an $\abf$.

\end{proof}

\begin{proposition}
The class of $(\alpha,\beta)$-representable posets is closed under ultraroots.
\end{proposition}
\begin{proof}
Suppose $\prod_U P$ is $(\alpha,\beta)$-representable and let $p,q\in P$ with $p\not\leq q$. Then there is an $\abf$ $\Gamma$ of $\prod_U P$ with $[\bar{p}]\in\Gamma$ and $[\bar{q}]\not\in \Gamma$. So by lemma \ref{L;filterFromPower} $\gamma$ is an $\abf$ of $P$ with $p\in \gamma$ and $q\not\in \gamma$ and thus $P$ is $(\alpha,\beta)$-representable.
\end{proof}

\begin{theorem}
If $\alpha,\beta\leq \omega$ then the class of $(\alpha,\beta)$-representable posets is elementary.
\end{theorem}
\begin{proof}
We've shown the class is closed under ultraproducts and ultraroots, and closure under isomorphism is trivial, so theorem \ref{T;elemcond} applies.
\end{proof}

Following the suggestion of the referee we use example \ref{E;notN} to show that the class of representable posets is not finitely axiomatizable.

\begin{theorem}
If $\alpha,\beta\leq \omega$ and either $\alpha=\omega$ or $\beta=\omega$ then the class of $(\alpha,\beta)$-representable posets is not finitely axiomatizable.
\end{theorem} 
\begin{proof}
Assume $\beta=\omega$ and note that a class is basic elementary if and only if both it and its complement are elementary. For each $4\leq n < \omega$ let $P_n$ be a poset as in example \ref{E;notN}, let $U$ be a non-principal ultrafilter of $\wp(\omega)$ and consider $\prod_U P_n$. For $2\leq k<\omega$ let $\phi_k$ be a first order sentence equivalent to there being no non-trivial suprema of size less than or equal to $|k|$ defined in $P$. Then $P_n\models \phi_k$ whenever $k\leq n-2$. So $\prod_U P_n \models \phi_k$ for all $k$, and hence is representable by taking principal up-sets. So the complement of the class of $(\alpha,\omega)$-representable posets is not closed under ultraproducts and is therefore not elementary by theorem \ref{T;elemcond}. Thus the class of $(\alpha,\omega)$-representable posets is not finitely axiomatizable. The argument for $(\omega,\beta)$ is dual. 
\end{proof}

Note that it is not known whether the classes of $(\alpha,\beta)$-representable posets are finitely axiomatizable when $2<\alpha,\beta<\omega$. This is equivalent to their complement classes being closed under ultraproducts. 
\end{section}

\begin{section}{Complete representations for posets}\label{S;comp}
In section \ref{S;elem} we proved indirectly that when $\alpha,\beta\leq\omega$ the representation classes for $(\alpha,\beta)$ are elementary. The proof used does not generalize to the cases when one or both of $\alpha$ and $\beta$ are greater than $\omega$. Indeed, given that first order logic is not expressive enough to distinguish between the cardinalities of infinite sets it should not be surprising that we cannot say very much about the situations where $\alpha$ and/or $\beta$ are larger than $\omega$.

Note that for some classes of posets complete representability \emph{can} be axiomatized in first order logic. Boolean algebras, for example, are completely representable if and only if they are atomic \cite[theorem 2.21]{HirHod02}, \cite[corollary 1]{Abi71}. Since all Boolean algebras are representable, any non-atomic Boolean algebra serves as an example of a poset that is representable but not completely representable. 

Theorem \ref{T;notElem} below, which also appears as \cite[theorem 3.2]{EgrHir12}, shows that the extra structure of Boolean algebras is necessary here as the class of completely representable posets (distributive lattices to be exact) is not closed under elementary equivalence and hence cannot be elementary. It is conjectured in \cite{EgrHir12} that the classes of $(\omega,C)$ and $(C,\omega)$-representable distributive lattices are not elementary. We conjecture here that the classes of $(\alpha,C)$ and $(C,\beta)$-representable posets are not elementary for all $\alpha,\beta\geq 3$.  

\begin{theorem}\label{T;notElem}
The class of completely representable posets is not closed under elementary equivalence.
\end{theorem}
\begin{proof}
The lattice $L=[0,1]\subseteq \mathbb{R}$ is not completely representable (it contains no non-trivial $C$-filters), however the lattice $L'=[0,1]\cap \mathbb{Q}$ \emph{is} completely representable as for every irrational $r$ the set $\{ a\in L'\colon a> r\}$ is a $C$-filter. $L$ and $L'$ are elementarily equivalent as $\mathbb{R}$ and $\mathbb{Q}$ are.
\end{proof} 

The situation seems somewhat negative at this point, but we can say something positive about the various classes of completely representable posets provided we relax our criteria a little. We use arguments adapted from those for distributive lattices given in \cite[section 3]{EgrHir12}.    

\begin{definition}[Pseudoelementary class]\label{D;pseudoClass}
A class $\sC$ of $\sL$-structures is pseudoelementary if there is a language $\sL'$ with $\sL\subseteq \sL'$, and an $\sL'$-theory $\tau$ such that $\sC$ is the class of all $\sL$-reducts of the class of all models of $\tau$.
\end{definition}

\begin{proposition}\label{P;pseudoelem}
A class $\sC$ of $\sL$-structures is pseudoelementary if and only if there are
\begin{enumerate}
\item a two-sorted language $\sL^+$, with disjoint sorts $\bbA$ and $\bbS$, containing $\bbA$-sorted copies of all symbols of $\sL$, and
\item an $\sL^+$ theory $\tau$  
\end{enumerate}
with $\sC=\{\gA^{\bbA}\upharpoonright_{\sL}\colon \gA\models \tau \}$, where $\gA$ is an $\sL^+$-structure, $\gA^{\bbA}$ is the structure in the sublanguage of $\sL^+$ containing only $\bbA$-sorted symbols whose domain contains the $\bbA$-sorted elements of $\gA$, and $\gA^{\bbA}\upharpoonright_{\sL}$ is the $\sL$ reduct of $\gA^{\bbA}$ obtained by identifying the symbols of $\sL$ with their $\bbA$-sorted counterparts in $\sL^+$.    
\end{proposition}

Proposition \ref{P;pseudoelem} follows from \cite[Proposition 9.9]{HirHod02} and the discussion at the start of \cite[Chapter 9.4.3]{HirHod02}. We are now in a position to prove the main results of this section.

\begin{theorem}\label{T;pospseud}
Let $\beta\leq \omega$. Then the class of $(C, \beta)$-representable posets is pseudoelementary.
\end{theorem}  

To prove this we proceed as follows (note that the case where $\beta=2$ is trivial so we assume $2<\beta$). Define $\sL$ to be the language of posets in first order logic, so $\sL=\{\leq\}$, and define $\sL^+$ to be the two-sorted language $\{\leq,\in,\fin\}$, with sorts $\bbA$ and $\bbS$, and $\leq$ and $\in$ being binary and $\fin$ being unary. In $\sL^+$ the first argument of $\in$ is $\bbA$ sorted and the second is $\bbS$ sorted, while $\leq$ takes both its arguments from $\bbA$, and $\fin$ takes its argument from $\bbS$. Define additional predicates as follows:
\begin{itemize}
\item $\subseteq(\esS,\etS)\iff\forall\eaA((\eaA\in\esS)\rightarrow(\eaA\in\etS))$

\item $\lb(\eaA,\esS)\iff\Big(\forall \ebA\big((\ebA\in \esS)\rightarrow (\eaA\leq \ebA)\big)\wedge \forall\ecA\big( \forall\edA((\edA\in\esS)\rightarrow (\ecA\leq\edA))\rightarrow (\ecA\leq \eaA)\big)\Big)$

\item $\ub(\eaA,\esS)\iff\Big(\forall \ebA\big((\ebA\in \esS)\rightarrow (\ebA\leq \eaA)\big)\wedge \forall\ecA \big(\forall \edA((\edA\in\esS)\rightarrow (\edA\leq\ecA))\rightarrow (\eaA\leq \ecA)\big)\Big)$
\item $\boU(\esS)\iff\forall\eaA\ebA\Big(\big((\eaA\in\esS)\wedge(\eaA\leq\ebA)\big)\rightarrow(\ebA\in\esS)\Big)$

\item $\bfC(\esS)\iff\forall\etS\eaA\Big(\big((\etS\subseteq\esS)\wedge \lb(\eaA,\etS) \big)\rightarrow (\eaA\in\esS)\Big)$
\item $\boA(\esS)\iff \forall\etS\eaA\Big(\big(\fin(\etS)\wedge\ub(\eaA,\etS)\wedge(\eaA\in\esS)\big)\rightarrow \exists \ebA\big((\ebA\in\etS)\wedge(\ebA\in\esS)\big)\Big)$
\end{itemize}
So $\subseteq$ corresponds to set inclusion, $\lb$ and $\ub$ to the concepts of greatest lower bound and least upper bound respectively, $\ub$ corresponds to the idea of a non-empty up-closed set, and $\bfC$ and $\boA$ will define certain completeness and primality properties respectively. We aim to define a theory so that we can capture the concept of a $(C,\beta)$-filter using these predicates.

Let $T$ be the $\sL$ theory defining partially ordered sets and define an $\sL^+$ extension $T^+$ of $T$ by adding the following axioms:
\begin{enumerate}
\item[$T^+_1$:] $\forall \eaA\ebA\Big((\eaA\not\leq\ebA)\rightarrow \exists\esS\big(\boU(\esS) \wedge \bfC(\esS)\wedge \boA(\esS)\wedge (\eaA\in\esS)\wedge(\ebA\notin\esS)\big ) \Big)$
\item[$T^+_2$:] $\forall\eaA\exists\esS\forall\ebA\Big((\eaA<\ebA)\leftrightarrow (\ebA\in\esS) \Big)$

Note that here $<$ is defined using $\leq$ and means strictly less than.

\item[$T^+_3$:] $\forall\esS\etS\exists\euS\forall\eaA\Big(\big((\eaA\in\esS)\wedge(\eaA\in\etS)\big)\leftrightarrow (\eaA\in\euS)  \Big)$

And also adding for each $2\leq n<\beta$ the axiom defined by

\item[$S_n$:] $\forall a_1^{\bbA}...\forall a_n^{\bbA}\exists \esS\big(\bw_{i=1}^n(a_i^{\bbA}\in\esS)\wedge \fin(\esS)\wedge \forall \ebA(\ebA\in\esS\rightarrow \bv_{i=1}^n(a_i^{\bbA}=\ebA)) \big)$.
\end{enumerate}
The first of these axioms forces its models to be separated by up-sets with the limited completeness and primality properties defined by $\bfC$ and $\boA$. The set $\{S_n:n<\beta\}$ together demand that the $\fin$ predicate `sees' all the sets it's supposed to, i.e. whenever an appropriately sized set of $\bbA$ sorted elements exists it defines a corresponding $\bbS$ sorted object. The role of $T^+_2$ and $T^+_3$ is to ensure that in models of $T^+$ sufficient $\bbS$ sorted elements are present to guarantee separation by $(C,\beta)$-filters. This is made precise in the following lemma.

\begin{lemma}\label{L;pseudo}
The class $\{M^{\bbA}\upharpoonright_{\sL}\colon M\models T^+ \}$ of $\sL$-reducts of models of $T^+$ is precisely the class of $(C,\beta)$-representable posets.
\end{lemma}
\begin{proof}
If $P$ is a  $(C,\beta)$-representable poset we can use $(P,\wp(P),\leq,\in,small)$ to model $T^+$ by the generalized version of theorem \ref{T;rep} ($small$ holds of a set $X$ if and only if $|X|<\beta$). Conversely if $A= M^{\bbA}\upharpoonright_{\sL}$ for some model $M$ of $T^+$ then it is clearly a poset and by $T^+_1$ the interpretation of the $\in$ predicate can be used to naturally define a separating set $K$ of up-sets of $A$ for which conditions $\bfC$ and $\boA$ hold. Moreover, if $Z\in K$, $Y\subseteq A$ and $|Y|<\beta$, then $\boA$ together with the $S_n$ axioms imply that if $\bv Y$ exists in $A$ and $\bv Y \in Z$ then there is $y\in Y\cap Z$. So to complete the proof let $S\in K$, let $T\subseteq S$, and suppose $x=\bw T$ exists in $A$. 

If $x=\bw(\{y\in A:y>x\})$ then one of two things must occur. First we may have $T\subseteq \{y\in A:y>x\}\cap S$. In this case $x=\bw T \geq z$ for all lower bounds $z$ of $\{y\in A:y>x\}\cap S$. Moreover, $x$ \textit{is} a lower bound for $\{y\in A:y>x\}\cap S$, so $x=\bw(\{y\in A:y>x\}\cap S)$. By axioms $T^+_2$ and $T^+_3$, and the definitions of the set $K$ and the predicate $\bfC$ we must have $x\in S$. Alternatively, if $T\not\subseteq \{y\in A:y>x\}\cap S$ then $x\in T$, and thus in $x\in S$ and we are done. 

Suppose instead that $x\neq\bw(\{y\in A:y>x\})$ and $x\notin T$. Then $T\subseteq \{y\in A:y>x\}$. As $x$ is clearly a lower bound for $\{y\in A:y>x\}$ if it is not the greatest lower bound there must be $z\in A$ with $z\leq y$ for all $y>x$ but $z\not\leq x$, but then if $z$ would also be a lower bound for $T$, which would contradict the assumption that $x=\bw T$. We conclude that $S$ is a $(C,\beta)$-filter, and in light of theorem \ref{T;rep} that $A$ is a $(C,\beta)$-representable poset. 
\end{proof}

Lemma \ref{L;pseudo} along with proposition \ref{P;pseudoelem} concludes the proof of theorem \ref{T;pospseud}. A dual argument gives the following result.

\begin{theorem}
Let $\alpha\leq \omega$. Then the class of $(\alpha, C)$-representable posets is pseudoelementary.
\end{theorem}

Finally, by replacing the predicate $\boA$ with 
\begin{equation*}
\bar{\textbf{C}}(\esS)\iff\forall\etS\eaA\Big(\big((\etS\subseteq\esS)\wedge \ub(\eaA,\etS) \big)\rightarrow (\eaA\in\esS)\Big)
\end{equation*} 
and proceeding with a similar argument replacing $T^+_1$ with
\begin{align*}
\bar{T}^+_1: \forall \eaA\ebA\Big(&(\eaA\not\leq\ebA)\rightarrow \exists\esS\big(\boU(\esS) \wedge \bfC(\esS)\\
&\wedge (\eaA\in\esS)\wedge(\ebA\notin\esS)\\
& \wedge \exists\etS(\forall\ecA(\ecA\notin\esS\leftrightarrow \ecA\in \etS)\wedge\bar{\textbf{C}}(\etS))\big)\Big)
\end{align*}
and adding the single axiom
\begin{equation*}
T^+_4: \forall\eaA\exists\esS\forall\ebA\Big((\eaA>\ebA)\leftrightarrow (\ebA\in\esS) \Big)
\end{equation*}

in place of the $S_n$ axioms we can prove our final result.

\begin{theorem}
The class of completely representable posets is pseudoelementary.
\end{theorem}
\begin{proof}
The argument is essentially the same as that for theorem \ref{T;pospseud}. The main difference is that the axiom $\bar{T}^+_1$ demands that the poset elements be separated by a set of upsets that are complete with respect to the $\bfC$ predicate and whose complements are complete with respect to the $\bar{\textbf{C}}$ predicate. The axiom $T^+_4$ plays a similar role to $T^+_2$ in the original proof and along with $T^+_3$ ensures that $\bar{\textbf{C}}$ translates to actual join-completeness in any model. The result of this is that in any model the poset elements are separated by a set of meet-complete up-closed sets whose complements are  join-complete. By corollary \ref{C;FI} this is equivalent to them being $C$-filters so we are done.
\end{proof}

Note that this argument does not require the $S_n$ axioms so the axiomatization is finite, unlike the situation where $\alpha$ or $\beta$ is $\omega$.

\end{section}

\end{document}